%% file: 0000.tex
\documentclass[10pt,twocolumn,a4paper]{article}

\usepackage{times}
\usepackage{epsfig}
\usepackage{graphicx}
\usepackage{amsmath}
\usepackage{amssymb}

\usepackage{authblk}

\input{preamble}

\usepackage[top=2.54cm, bottom=4.13cm, left=1.75cm, right=1.75cm]{geometry}

\usepackage[breaklinks=true,bookmarks=false,colorlinks=true,linkcolor=blue]{hyperref}

\def\myparagraph#1{\vspace{6.8pt}\noindent{\bf #1~~}}

\begin{document}

\title{Combinatorial persistency criteria for multicut and max-cut}
\date{}

\author[1]{Jan-Hendrik Lange}
\author[1,2,3]{Bjoern Andres}
\author[1]{Paul Swoboda}

\affil[1]{Max Planck Institute for Informatics, Saarbr\"ucken}
\affil[2]{Bosch Center for Artificial Intelligence, Renningen}
\affil[3]{University of T\"ubingen}

\maketitle

\begin{abstract}
In combinatorial optimization, partial variable assignments are called persistent if they agree with some optimal solution.
We propose persistency criteria for the multicut and max-cut problem as well as fast combinatorial routines to verify them.
The criteria that we derive are based on mappings that improve feasible multicuts, respectively cuts.
Our elementary criteria can be checked enumeratively.
The more advanced ones rely on fast algorithms for upper and lower bounds for the respective cut problems and max-flow techniques for auxiliary min-cut problems.
Our methods can be used as a preprocessing technique for reducing problem sizes or for computing partial optimality guarantees for solutions output by heuristic solvers.
We show the efficacy of our methods on instances of both problems from computer vision, biomedical image analysis and statistical physics.  
\end{abstract}

\input{section-intro}
\input{section-related-work}
\input{section-problem}

\input{section-improving-mappings}

\input{section-persistency-criteria}
\input{section-algorithms}
\input{section-experiments}

\section{Conclusion}

We have presented combinatorial persistency criteria for the \multicut{} and \maxcut{} problem.
Moreover, we have devised efficient algorithms to check our criteria.
For \multicut{} our method achieves a substantial improvement over prior work when evaluated on common benchmarks as well as practical instances.
For \maxcut{} we are, to the best of our knowledge, the first to propose an algorithm that computes persistent variable assignments for the general problem.
For the special case of QPBO problems, our method matches the performance of prior work on some instances.
Our results demonstrate the feasibility of computing persistent variable assignments for NP-hard graph cut problems in practice.
Besides acquiring partial optimality guarantees, our approach is a helpful tool for shrinking problem sizes and thus essential toward identifying globally optimal solutions.

{\small
\bibliographystyle{ieee}
\bibliography{0000}
}

\input{appendix}

\end{document}

%% file: preamble.tex
\usepackage{microtype}

\usepackage[shortlabels, inline]{enumitem} % customizable enumeration/itemization environments

\usepackage{booktabs}
\usepackage[retainorgcmds]{IEEEtrantools}
\usepackage{cases}

\usepackage{algorithm}
\usepackage{algorithmic}

\usepackage[normalem]{ulem}

\usepackage{amssymb}
\usepackage{amsmath}
\usepackage{amsthm}
\usepackage{bbm} % more black-board bold

\usepackage{pgfplots}
\usepackage{pgfplotstable}
\usepackage{tikz}
    \definecolor{myblue}{HTML}{3333CC}
    \definecolor{mylightblue}{HTML}{CCCCFF}
    \definecolor{myorange}{HTML}{FF6600}
    \definecolor{myred}{HTML}{CC3333}
    \definecolor{mygreen}{HTML}{11AA11}
    \definecolor{mylightgray}{HTML}{E0E0E0}
    \usetikzlibrary{decorations.pathmorphing}
    \tikzstyle{vertex}=[circle, draw, fill=white, inner sep=0pt, minimum width=1ex]
    \tikzset{every picture/.append style={baseline,scale=1.1}}
    \tikzstyle{cut-edge}=[dotted]
    \tikzstyle{base}=[line width=1.2pt]
    \tikzstyle{lifted}=[blue]

\usetikzlibrary{external}
\providecommand{\R}{\mathbb{R}}

\providecommand{\mc}{\mathsf{MC}}

\providecommand{\cut}{\mathsf{CUT}}

\providecommand{\1}{\mathbbmss{1}}
\newcommand{\closure}[2][2]{{}\mkern#1mu\overline{\mkern-#1mu#2\mkern-1mu}}

\providecommand{\la}{\langle}
\providecommand{\ra}{\rangle}

\providecommand{\abs}[1]{\lvert #1 \rvert}

\providecommand{\maxcut}{\textsc{max-cut}}
\providecommand{\multicut}{\textsc{multicut}}

\theoremstyle{plain}
\newtheorem{thm}{Theorem}
\newtheorem{lem}{Lemma}

\newtheorem{cor}{Corollary}

\theoremstyle{definition}
\newtheorem{defn}{Definition}
\newtheorem{fact}{Fact}

\newtheorem{assump}{Assumption}

\theoremstyle{remark}

\DeclareMathOperator{\id}{id}
\DeclareMathOperator*{\argmin}{arg min}

\DeclareMathOperator{\sign}{sign}

%% file: section-intro.tex
% !TeX root = 0000.tex

\section{Introduction}
Partitioning graphs into meaningful clusters is a fundamental problem in combinatorial optimization with numerous applications in computer vision, biomedical image analysis, machine learning, data mining and beyond.
The \multicut{} problem (a.k.a.\ correlation clustering) and \maxcut{} problem are arguably among the most well-known combinatorial optimization problems for partitioning graphs.
They enable graph clustering purely based on costs between pairs of nodes and are thus commonly employed to model image processing and segmentation tasks occurring in computer vision \cite{Rother2007, Andres2011, Keuper2015a, Insafutdinov2016, Beier2017}.
The following factors contribute to the significance of the \multicut{} and \maxcut{} problem: The former allows for a graph clustering formulation that determines the number of clusters as part of the optimization process. The latter is essentially equivalent to binary quadratic programming, which has a variety of applications in image processing.
However, as computer vision models are typically large-scale, standard solution techniques based on solving LP-relaxations do not scale well enough and are thus inapplicable.
Even more so, finding globally optimal solutions with branch-and-cut is infeasible with off-the-shelf commercial solvers.
Hence, the need arises for developing specialized heuristic solvers that output high-quality solutions for real-world problems, despite the worst-case NP-hardness of the \multicut{} and \maxcut{} problem.
Unfortunately, although heuristic solvers often achieve a good empirical performance, they usually come without any optimality guarantees.
Specifically, even if large parts of the variable assignments computed by a heuristic agree with globally optimal solutions, such optimality is not recognized.

In this work we consider combinatorial techniques for the \multicut{} and \maxcut{} problem by which we can efficiently find \emph{persistency} (a.k.a.\ partial optimality). 
Persistent variable assignments come with a certificate that proves their agreement with a globally optimal solution.
The potential benefits are twofold:
(i) After running a primal heuristic, we can compute certificates which show that some variables are persistent.
(ii) Even before running a heuristic, we may determine in a preprocessing step persistent variable assignments.
In either case, the problem size can be reduced.
In the first case, a subsequent optimization with exact solvers is accelerated.
In the second case, possibly also the runtime of a heuristic algorithm is reduced and the solution quality improved.

A joint treatment of the \multicut{} and \maxcut{} problem seems instructive, since many criteria have a similar formulation and are based on analogous arguments.
For the \maxcut{} problem we offer, to our knowledge, a novel approach for computing persistent variable assignments.
For the \multicut{} problem our empirical evidence suggests that our method offers substantial improvement over prior work on persistency.
%Therefore, our results are relevant for large scale problems that current heuristics can barely handle.
Our empirical results are most significant for very large scale problems which current heuristics can barely handle, e.g.\ in biomedical image segmentation~\cite{Beier2017}.
By reducing problem size via persistency, our method enables high quality solutions in such cases.

The paper is organized as follows.
In Section \ref{sec:related-work} we review the related work.
In Section \ref{sec:problem} we introduce the \multicut{} and \maxcut{} problem mathematically in a shared compact formulation.
In Section \ref{sec:improving-mappings} we recap the concept of improving mappings in the context of persistency. Further, we introduce fundamental building blocks for the construction of improving mappings for the \multicut{} and \maxcut{} problem.
In Section~\ref{sec:persistency-criteria} and~\ref{sec:algorithms} we present our combinatorial persistency criteria and devise algorithms to check them.
Finally, in Section~\ref{sec:experiments} we evaluate our methods in numerical experiments on instances from the literature and compare to related work.
The more technical proofs for our results are provided in the appendix. In the appendix we also present technical improvements of our persistency criteria that were omitted from the main paper for the sake of clarity.

%% file: section-related-work.tex
% !TeX root = 0000.tex

\section{Related work} \label{sec:related-work}

Persistency for Markov Random Fields (MRF) and, as a special case, for the binary quadratic optimization problem (a.k.a.\ Quadratic Pseudo-Boolean Optimization (QPBO)), has been well studied.
It was observed in~\cite{Nemhauser1975} that a natural LP-relaxation of the stable set problem has the \emph{persistency} property: All integral variables of LP-solutions coincide with a globally optimal one.
This result has been transferred to QPBO~\cite{Hammer1984,Boros2002,Boros2008} and extended in~\cite{Wang2009} to find relational persistency, i.e. showing that some pairs of variables must have the same/different values.
For higher order binary unrestricted optimization problems, the concept of roof duality can be extended to obtain further persistency results~\cite{Rother2007,Kahl2012,Kolmogorov2012}.
Going beyond the basic LP-relaxation for QPBO, persistency certificates involving tighter LP-relaxations for higher order polynomial 0/1-programs that do not possess the persistency property (i.e.\ integral variables need not be persistent) have been studied in~\cite{Adams1998}.

For general MRFs, criteria that can be elementarily checked include Dead End Elimination (DEE) \cite{Desmet1992}.
More powerful techniques generalizing DEE that still can be used for fast preprocessing can be found in~\cite{Wang2016}.
The MQPBO method~\cite{Kohli2008} consists of transforming multilabel MRFs to the QPBO problem and persistency results from QPBO can subsequently be used to obtain persistency for the original multilabel MRF.
Persistency criteria for the multilabel Potts problem that can be efficiently checked with max-flow computations have been developed in~\cite{Kovtun2003,Kovtun2011} and refined in~\cite{Gridchyn2013}.
More powerful criteria based on LP-relaxations have been proposed for the multilabel Potts problem in~\cite{Swoboda2013} and in~\cite{Shekhovtsov2014,Swoboda2016,Shekhovtsov2017} for general discrete MRFs.
An in-depth exposition of the concept of improving mappings that is used implicitly or explicitly for all of the above MRF criteria can be found in~\cite{Shekhovtsov2013}. 
A comprehensive theoretical discussion and comparison of the above persistency techniques can be found in~\cite{Shekhovtsov2016}.

There has been, to our knowledge, less work on persistency for the \multicut{} and \maxcut{} problem.
For \multicut, the works~\cite{Alush2012,Lange2018} proposed simple persistency criteria that allow to fix some edge assignments.
We are not aware of any persistency results for \maxcut.
Also it is not easily possible to transfer persistency results from QPBO to \maxcut, even though there exist straightforward transformations between these two problems.
The underlying reason is that the transformation from \maxcut{} to QPBO introduces symmetries which current persistency criteria cannot handle.
More specifically, known persistency criteria rely on an improving mapping, but in symmetric instances it is always possible to map a labeling to an equivalent one with the same cost by exploiting symmetries. 
Consequently, fixed-points of improving mappings, which amount to persistent variables, cannot be found.
For the closely related (yet polynomial-time solvable) \textsc{min-cut} problem, a family of persistency criteria were proposed in~\cite{Padberg1990,Henzinger2018}.
They directly translate to the \maxcut{} problem and we derive them as special cases in our study below.

The more involved constraints describing the \multicut{} and \maxcut{} problem make it difficult to directly transfer some of the powerful persistency techniques that are available for MRFs.
In our work we show how the framework of improving mappings developed in~\cite{Shekhovtsov2013} can be used to derive persistency criteria for combinatorial problems with more complicated constraint structures, such as the \multicut{} and \maxcut{} problem, once a class of mappings that act on feasible solutions is identified.
Specifically, we show that the known \multicut{} persistency criteria from~\cite{Lange2018} and the persistency criteria from~\cite{Henzinger2018} (transferred to the \maxcut{} problem) can be derived in our theoretical framework.
Moreover, we define more powerful criteria that can find significantly more persistent variables, as shown in the experimental Section~\ref{sec:experiments}, yet can be evaluated efficiently.
We believe that our approach of composing improving mappings from elementary mappings is instructive in the search for more persistency criteria.

%% file: section-problem.tex
% !TeX root = 0000.tex

\section{Multicut and max-cut} \label{sec:problem}

Let
\begin{align}
\label{eq:problem}
\min \; & \la \theta, x\ra \quad \text{ s.t. } \quad x \in X \tag{P}
\end{align}
with $X \subseteq \{0,1\}^m$ be a linear combinatorial optimization problem.  
In this paper, we study specific instances of \eqref{eq:problem} known as the \multicut{} and the \maxcut{} problem, which are introduced mathematically in this section.
To this end, let $G = (V,E,\theta)$ be a weighted graph, where $\theta \in \R^E$.
We distinguish non-negative and negative edges via $E = E^+ \cup E^-$ with $E^+ = \{e \in E \mid \theta_e \geq 0\}$ and $E^- = \{e \in E \mid \theta_e < 0\}$.
For any two disjoint subsets of vertices $U, W \subseteq V$ let $\delta(U,W) = \{ uw \in E \mid u \in U, w \in W\}$ denote the set of edges between $U$ and $W$. Further, we write $\delta(U) := \delta(U, V \setminus U)$.
For any subgraph $H = (V_H, E_H)$ of $G$ we may identify $H$ with $E_H$ and write $e \in H$ instead of $e \in E_H$.

\begin{defn}[Multicuts and Cuts]
Let $(U_1,\ldots,U_k)$ be a partition of $V$, i.e. $U_1 \cup \ldots \cup U_k = V$ and $U_i \cap U_j = \varnothing$ for $i\neq j$. The set of edges $M$ between any pair of components of the partition, defined by
\begin{align*}
M = \bigcup_{1 \leq i < j \leq k} \delta(U_i, U_j),
\end{align*}
is called a \emph{multicut} of $G$.
If $k = 2$, then $M = \delta(U_1) = \delta(U_2)$ is called a \emph{cut} of $G$.
For any set of edges $F \subseteq E$ define the incidence vector $\1_F \in \{0,1\}^E$ of $F$ via
\begin{align*}
(\1_F)_e = \begin{cases} 1 & \text{if } e \in F \\ 0 & \text{else.} \end{cases}
\end{align*}
We write
\begin{align*}
\mc := & \Big \{ \1_M \mid M \text{ multicut of } G \Big \}, \\
\cut := & \left \{ \1_{\delta(U)} \mid U \subseteq V \right \} \subseteq \mc
\end{align*}
for the set of incidence vectors of multicuts, respectively cuts of $G$.
\end{defn}

\myparagraph{Multicut.}
The \multicut{} problem is to find a multicut of minimum weight w.r.t.\ $\theta$ and can be written as an instance of~\eqref{eq:problem} as follows:
\begin{align}
\label{eq:multicut}
\min \; & \la \theta, x\ra \quad\text{s.t.}\quad x \in \mc. \tag{P$_{\mc}$}
%\mc := &
%\left\{x \in \{0,1\}^E \mid 
%\begin{array}{c}
%\displaystyle x_e - \sum_{e' \in C \backslash \{e\}} x_{e'} \leq 0 \\
%\forall \text{ cycles } C, e\in C \\
%\end{array}
%\right\}
\end{align}
%The inequalities in~\eqref{eq:multicut} require every cycle to have either zero or at least two cut edges. 
%They are a subset of the inequalities defining~\eqref{eq:max-cut} (since every cut is also a multicut).

\myparagraph{Max-Cut.}
The \maxcut{} problem is to find a cut $\delta(U)$, $U \subseteq V$, of maximum weight (or equivalently of minimum weight for $-\theta$).
After setting $\theta \leftarrow -\theta$ it can be written as an instance of~\eqref{eq:problem} as follows:
\begin{align}
\label{eq:max-cut}
\min \; & \la \theta, x\ra \quad \text{ s.t. } \quad x \in \cut. \tag{P$_{\cut}$}
%\cut := & \left\{x \in \{0,1\}^E \mid
%\begin{array}{c}
% \displaystyle \sum_{e \in C\backslash F} x_e - \sum_{e \in F} x_e \leq \abs{F} - 1 \\
%\forall \text{ cycles } C, F\subset C, \abs{F} \text{ odd}\\
%\end{array}
%\hspace{-0.5em}
%\right\}
\end{align}
Note that we use $\min$ instead of $\max$ to conform to~\eqref{eq:problem}.
%The inequalities in~\eqref{eq:max-cut} require every cycle to have an even number of cut edges.

%% file: section-improving-mappings.tex
% !TeX root = 0000.tex

\section{Improving mappings} \label{sec:improving-mappings}

In this section, we introduce improving mappings as a concept to derive partial optimality results and define elementary building blocks to construct improving mappings for the \multicut{} and \maxcut{} problem.

%\begin{defn}[Persistency]
%Let $i \in [m]$ and $\beta \in \{0,1\}$. The variable $x_i$ is called (weakly) \emph{$\beta$-persistent} for problem \eqref{eq:problem} if $x^*_i = \beta$ in some optimal solution $x^*$ of \eqref{eq:problem}.
%\end{defn}

\begin{defn}[\cite{Shekhovtsov2014}]
A mapping $p \colon X \rightarrow X$ with the property
\begin{equation*}
  \la \theta, p(x) \ra \leq \la \theta, x \ra \quad \forall x \in X
\end{equation*}
is called \emph{improving mapping}.
\end{defn}
%We can require w.l.o.g.\ $p$ to be idempotent, e.g.\ $p = p \circ p$ (otherwise take the least power $p^n$ that is idempotent).

An improving mapping $p$ that maps some variable $i \in [m]$ to a fixed value $\beta$ provides \emph{persistency} (a.k.a.\ partial optimality):
For each feasible element $x \in X$, a better one is obtained by applying $p$ to $x$ and thus fixing $x_i = \beta$.

\begin{lem}[Persistency]
\label{lem:persistency}
Let $p \colon X \to X$ be an improving mapping and $\beta \in \{0,1\}$. If
\begin{align*}
p(x)_i = \beta \quad \forall x \in X,
\end{align*}
then $x^*_i = \beta$ in some optimal solution $x^*$ of \eqref{eq:problem}.
\end{lem}

\begin{proof}
Let $x$ be an optimal solution of \eqref{eq:problem}. Then $x^* = p(x)$ is also optimal and $x^*_i = \beta$.
\end{proof}

There are two trivial improving mappings:
(i)
The identity mapping $\id \colon x \mapsto x$. It does not provide any persistency at all.
(ii) 
The mapping $p^* \colon x \mapsto x^*$ that maps any $x$ to a fixed optimal solution $x^* \in \argmin_{x \in X} \la \theta, x\ra$.
This mapping obviously provides the maximal persistency, but for NP-hard problems it is generally intractable to compute $x^*$.

We are hence interested in a middle ground: We want to find improving mappings that fix as many variables as possible (unlike $\id$) but that are computable in polynomial time (unlike $p^*$).
This allows us to simplify the original problem~\eqref{eq:problem} by fixing the persistent variables.
For the \multicut{} problem we can contract those edges that can be persistently set to $0$, which allows to shrink the underlying graph.
For the \maxcut{} problem, however, any value for persistent variables can be exploited for contractions, as we show below.

\subsection{Elementary mappings}

In order to construct improving mappings for the \multicut{} and \maxcut{} problem, we employ the elementary mappings defined in this section.

\begin{figure}
\center
a)
\input{figures/multicut.tex}
b)
\input{figures/result-cut-mapping.tex}
c)
\input{figures/result-join-mapping.tex}
\caption{Illustration of elementary mappings. a) Original multicut $x \in \mc$ (solid lines) and connected region $U$ (dashed line). b) Result of cut mapping $p_{\delta(U)}(x)$. c) Result of join mapping $p_U(x)$.}
\label{fig:multicut-mappings}
\end{figure}
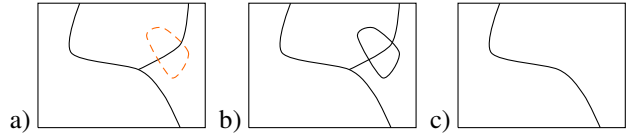

\begin{defn}[Multicut mappings]
Let $U \subseteq V$ be a set of nodes that induce a connected component of $G$.
\begin{itemize}[wide, labelwidth=!, labelindent=0pt]
\item[(i)]
The \emph{elementary cut mapping} $p_{\delta(U)}$ is defined as 
\begin{equation*}
p_{\delta(U)}(x) = x \vee \1_{\delta(U)}\,.
\end{equation*}
In other words, this means that $p_{\delta(U)}(x)_e = 1$ for all edges $e \in \delta(U)$ and $p_{\delta(U)}(x)_e = x_e$ otherwise.
\item[(ii)]
The \emph{elementary join mapping} $p_U$ is defined as
\begin{align}
\label{eq:elementary-join-mapping-definition}
p_U(x)_{uv} =
\begin{cases}
0, &  uv \in E(U) \\
0, & \exists uv\text{-path } P \text{ such that } \\ & \forall e \in E(P) : \\ &  x_e = 0 \text{ or } e \in E(U) \\
x_{uv}, & \text{otherwise}.
\end{cases}
\end{align}
\end{itemize}
\end{defn}

Intuitively, the elementary cut mapping $p_{\delta(U)}$ adds the cut $\delta(U)$ to the multicut defined by $x$. The elementary join mapping $p_U$ merges all components that intersect with $U$, cf.\ Figure \ref{fig:multicut-mappings}.
To show well-definedness of the elementary cut and join mapping rigorously, we need the following characterization of multicuts.

\begin{fact}[\cite{Chopra1993}]
A set $M \subseteq E$ is a multicut iff for every cycle $C$ of $G$ it holds that $\abs{M \cap C} \neq 1$.
\end{fact}

\begin{lem}[Well-definedness]
\label{lem:well-definedness}
The mappings $p_{\delta(U)}$ and $p_U$ are well-defined, i.e.
\begin{itemize}
\item[(i)] $p_{\delta(U)} \colon \mc \rightarrow \mc$ for any connected $U \subseteq V$
\item[(ii)] $p_U \colon \mc \rightarrow \mc$ for any connected $U \subseteq V$.
\end{itemize}
\end{lem}

The elementary mapping for the \maxcut{} problem exploits the well-known property of cuts that they are closed under taking symmetric differences (of edges).
\begin{fact}[\cite{Schrijver2003}] \label{fact:sym-diff-cut}
Let $x, y \in \cut$. Then $x \triangle y \in \cut$.
\end{fact}

In particular, since $x \mapsto x \triangle y$ is an involution (i.e.\ its own inverse) for any cut $y \in \cut$, it holds that $\cut \triangle y := \{ x \triangle y \mid x \in \cut\} = \cut$.
Given an instance of \maxcut{} defined by $G = (V,E,\theta)$ and a cut $y \in \cut$, this transformation of the feasible set corresponds to \emph{switching} the signs of $\theta_e$ for all $e \in E$ with $y_e = 1$ and adding the constant $\sum_{e \in E} \theta_e y_e$ to the objective value.
If $y$ is optimal for the original instance, then $y \triangle y = 0$ is optimal for the transformed instance.
Hence, whenever we want to compute persistency for $x_f = 1$, we can transform the instance to an equivalent one by applying the described switching for any cut that contains $f$ and then checking whether $x_f = 0$ holds persistently.

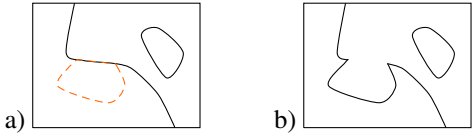
\begin{figure}
\center
a)
\input{figures/cut.tex}
\hspace{2em}
b)
\input{figures/result-sym-diff-mapping.tex}
\caption{Illustration of symmetric difference mapping. a) Original cut $x \in \cut$ (solid lines) and cut $\delta(U)$ (dashed orange line). b) Result of symmetric difference mapping $p^\triangle_{\delta(U)}(x)$.}
\label{fig:max-cut-mapping}
\end{figure}

\begin{defn}[Symmetric Difference Mapping]
Let $U \subseteq V$. %, $f \in \delta(U)$ and $\beta \in \{0,1\}$. 
The \emph{elementary symmetric difference mapping} $p^\triangle_{\delta(U)}$ w.r.t.\ $\delta(U)$ is defined as
\begin{align*}
p^\triangle_{\delta(U)}(x) = x \triangle \1_{\delta(U)} \, .
\end{align*}
In other words, this means that $p^\triangle_{\delta(U)}(x)_e = 1 - x_e$ for all edges $e \in \delta(U)$ and $p^\triangle_{\delta(U)}(x)_e = x_e$ otherwise. The symmetric difference mapping is well-defined because of Fact \ref{fact:sym-diff-cut}. See Figure \ref{fig:max-cut-mapping} for an illustration of $p^\triangle_{\delta(U)}$.
\end{defn}

%% file: figures/multicut.tex
\begin{tikzpicture}[scale=1]
    \draw[fill=none] plot coordinates
        {(0,0) (0, 1.5) (2, 1.5) (2,0) (0,0)};
        
    \draw[fill=none] plot[smooth] coordinates
    		{(0.5,1.5) (0.4,0.9) (1.2,0.7) (1.5,0.4) (1.7,0)};
    		
    	\draw[fill=none] plot[smooth] coordinates
		{(1.2,0.7) (1.7,1.0) (1.8,1.5)};
		
	\draw[fill=none, myorange, densely dashed] plot[smooth cycle] coordinates
		{(1.6,0.6) (1.8,0.9) (1.5,1.2) (1.3,1.1)};   			
\end{tikzpicture}

%% file: figures/result-cut-mapping.tex
\begin{tikzpicture}[scale=1]
    \draw[fill=none] plot coordinates
        {(0,0) (0, 1.5) (2, 1.5) (2,0) (0,0)};
        
    \draw[fill=none] plot[smooth] coordinates
    		{(0.5,1.5) (0.4,0.9) (1.2,0.7) (1.5,0.4) (1.7,0)};
    		
    	\draw[fill=none] plot[smooth] coordinates
		{(1.2,0.7) (1.7,1.0) (1.8,1.5)}; 
    		
	\draw[fill=none] plot[smooth cycle] coordinates
		{(1.6,0.6) (1.8,0.9) (1.5,1.2) (1.3,1.1)};
			
\end{tikzpicture}

%% file: figures/result-join-mapping.tex
\begin{tikzpicture}[scale=1]
    \draw[fill=none] plot coordinates
        {(0,0) (0, 1.5) (2, 1.5) (2,0) (0,0)};
        
    \draw[fill=none] plot[smooth] coordinates
    		{(0.5,1.5) (0.4,0.9) (1.2,0.7) (1.5,0.4) (1.7,0)};
			
\end{tikzpicture}

%% file: figures/cut.tex
\begin{tikzpicture}[scale=1.0]
    \draw[fill=none] plot coordinates
        {(0,0) (0, 1.5) (2, 1.5) (2,0) (0,0)};
        
    \draw[fill=none] plot[smooth] coordinates
    	{(0.5,1.5) (0.4,0.9) (0.6,0.81) (1.0,0.77) (1.2,0.7) (1.5,0.4) (1.7,0)};
		
	\draw[fill=none] plot[smooth cycle] coordinates
		{(1.6,0.6) (1.8,0.9) (1.5,1.2) (1.3,1.1)};
		
	\draw[fill=none, myorange, densely dashed] plot[smooth cycle] coordinates
		{(0.95,0.772) (1.0,0.75) (1.1,0.5) (0.9,0.3) (0.3,0.5) (0.5, 0.795) (0.6,0.81)}; 			
\end{tikzpicture}

%% file: figures/result-sym-diff-mapping.tex
\begin{tikzpicture}[scale=1.0]
    \draw[fill=none] plot coordinates
        {(0,0) (0, 1.5) (2, 1.5) (2,0) (0,0)}; 
    		
	\draw[fill=none] plot[smooth] coordinates
   		{(0.5,1.5) (0.4,0.9) (0.53,0.82)};
    		
    \draw[fill=none] plot[smooth] coordinates
    	{(1.0,0.77) (1.2,0.7) (1.5,0.4) (1.7,0)};
		
	\draw[fill=none] plot[smooth cycle] coordinates
		{(1.6,0.6) (1.8,0.9) (1.5,1.2) (1.3,1.1)};
		
	\draw[fill=none] plot[smooth] coordinates
		{(1.0,0.775) (1.1,0.5) (0.9,0.3) (0.3,0.5) (0.5, 0.79) (0.53,0.815)};
\end{tikzpicture}

%% file: section-persistency-criteria.tex
% !TeX root = 0000.tex

\section{Persistency criteria}
\label{sec:persistency-criteria}
In this section, we propose subgraph-based criteria for finding improving mappings.
We provide criteria for small connected subgraphs such as edges or triangles as well as criteria for general connected subgraphs.
In Section \ref{sec:algorithms}, we present efficient heuristic algorithms to check the subgraph criteria proposed in this section.

First consider the instructive special case of a single edge subgraph. The following criterion has been evaluated by~\cite{Lange2018} for the \multicut{} problem.

\begin{thm}[Edge Criterion]
\label{thm:single-edge-criterion}
Let $f \in E$ be an edge and $U \subseteq V$ be connected with $f \in \delta(U)$. Further, let $\beta = (1-\sign\theta_f) / 2$.
If
\begin{numcases}{}
\displaystyle \theta_f \geq \sum_{ e \in \delta(U) \setminus \{f\}} \abs{\theta_e}, & \textnormal{\ref{eq:problem}} $=$ \textnormal{\ref{eq:multicut}}, $\; \beta = 0$ \label{eq:multicut-single-edge-positive} \\[1ex]
\displaystyle \abs{\theta_f} \geq \sum_{ e \in \delta(U) \cap E^+} \theta_e, & \textnormal{\ref{eq:problem}} $=$ \textnormal{\ref{eq:multicut}}, $\; \beta = 1$ \label{eq:multicut-single-edge-negative} \\[1ex]
\displaystyle \abs{\theta_f} \geq \sum_{ e \in \delta(U) \setminus \{f\}} \abs{\theta_e}, & \textnormal{\ref{eq:problem}} $=$ \textnormal{\ref{eq:max-cut}} \label{eq:max-cut-single-edge}
\end{numcases}
then $x^*_f = \beta$ in some optimal solution $x^*$ of \eqref{eq:problem}.
\end{thm}

The proof of Theorem~\ref{thm:single-edge-criterion} relies on Lemma \ref{lem:persistency} by applying the mapping $p^\triangle_{\delta(U)}$, respectively $p_f \circ p_{\delta(U)}$, to improve any solution $x$ with $x_f \neq \beta$.
Simple candidates for $U$ are $\{u\}$ and $\{v\}$ where $f = uv$.
Checking these for every edge $f \in E$ can be done in linear time. All $u$-$v$-cuts can be checked at once by minimizing the right-hand sides of \eqref{eq:multicut-single-edge-positive} -- \eqref{eq:max-cut-single-edge} via max-flow techniques on the weighted graph $G^{\abs{\cdot}} = (V,E,\abs{\theta})$, respectively $G^+ = (V,E^+,\theta)$ for~\eqref{eq:multicut-single-edge-negative}. Note that the condition in~\eqref{eq:multicut-single-edge-negative} is less restrictive than~\eqref{eq:multicut-single-edge-positive}. Computing a Gomory-Hu tree \cite{Gomory1961} of $G^{\abs{\cdot}}$ or $G^+$ reduces the total computational effort of checking the criterion for all edges $f \in E$ to $\abs{V} - 1$ max-flow problems.
%Then for each $f \in E$ it is possible to determine the value of the minimum cut $\delta(U)$ containing $f$ w.r.t.\ edge weights $\abs{\theta}$.

\subsection{General subgraph criteria}
We give a technical lemma that allows to generalize the persistency criterion stated in Theorem~\ref{thm:single-edge-criterion}.
\begin{lem}
\label{lemma:subgraph-criterion}
Let $f \in E$ and $\beta \in \{0,1\}$.
Further, let $H = (V_H, E_H)$ be a connected subgraph of $G$ such that $e \in E_H$.
If for every $y \in \cut(H)$ with $y_f = 1 - \beta$, there exists a mapping $p^{y} \colon X \to X$ such that for all $x \in X$ with $x_{|E_H} = y$ we have
\begin{enumerate}[wide, labelwidth=!, labelindent=0pt]
\item[(i)]
$\la \theta, p^y(x) \ra \leq \la \theta, x \ra$
\item[(ii)]
$p^y(x)_f = \beta$,
\end{enumerate}
then $x^*_f = \beta$ in some optimal solution $x^*$.
\end{lem}
\begin{proof}
Condition (i) implies that the mapping $p \colon X \rightarrow X$ defined by
\begin{align*}
p(x) = \begin{cases} p^y(x) & \text{if } x_{|E_H} = y \\ x & \text{else} \end{cases}
\end{align*}
is improving.
Condition (ii) implies $p(x)_e = \beta$ for all $x$.
\end{proof}

Consider the special case when $H$ is a triangle subgraph.

\begin{cor}[Triangle Criterion]
\label{cor:triangle-criterion}
Let $\{uw,uv,vw\} \subset E$ be a triangle.
Let $U \subset V$ be such that $uv, uw \in \delta(U)$, and $W \subset V$ be such that $uw, vw \in \delta(W)$.
\begin{enumerate}[wide, labelwidth=!, labelindent=0pt]
\item[(i)]
If
\begin{align}
\theta_{uw} + \theta_{uv} \geq \sum_{e \in \delta(U) \backslash{\{uw,uv\}}} \abs{\theta_e} \label{eq:triangle-criterion-1} \\
\theta_{uw} + \theta_{vw} \geq \sum_{e \in \delta(W) \backslash{\{uw,vw\}}} \abs{\theta_e} \label{eq:triangle-criterion-2}
\end{align}
holds, then $x^*_{uw} = 0$ for some optimal solution of~\eqref{eq:max-cut}.
\item[(ii)]
If additionally
\begin{equation}
\theta_{uw} + \theta_{uv} + \theta_{vw} \geq \sum_{e \in \delta(\{u,v,w\}) \cap E^+} \theta_e \label{eq:triangle-criterion-3}
\end{equation}
holds, then $x^*_{uw} = 0$ for some optimal solution of~\eqref{eq:multicut}.  
\end{enumerate}
\end{cor}

A straightforward choice for the cuts in Corollary~\ref{cor:triangle-criterion} are $\delta(\{u\}), \delta(\{w\}), \delta(\{v,w\})$ and $\delta(\{u,v\})$, as depicted in Figure \ref{fig:subgraph-criteria} a).
It is possible to find better cuts w.r.t.\ costs $\abs{\theta}$, but we are not aware of any more efficient technique than to explicitly compute them via max-flow for every triangle (unlike computing a Gomory-Hu tree to evaluate the single edge criterion for all edges).

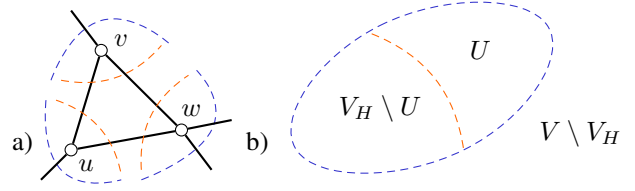
\begin{figure}
\flushright
a)
\input{figures/triangle-criterion.tex}
b)
\input{figures/subgraph-schematic.tex}
%\vspace{1ex}
\caption{a) The conditions presented in Corollary \ref{cor:triangle-criterion} compare the weights of inner cuts (\textcolor{myorange}{- \negthinspace -}) and outer cuts (\textcolor{myblue}{- \negthinspace -}) around the triangle $\{u,v,w\}$. b) The conditions \eqref{eq:subgraph-multicut-condition} and \eqref{eq:subgraph-max-cut-condition}, presented in Theorem~\ref{thm:multicut-subgraph-criterion} and \ref{thm:max-cut-subgraph-criterion}, compare the weights of the inner cut $\delta(U, V_H \setminus U)$ and the outer cut $\delta(V_H) = \delta(U, V \setminus V_H) \cup \delta(V_H \setminus U, V \setminus V_H)$.}
\label{fig:subgraph-criteria}
\end{figure}

We further apply Lemma~\ref{lemma:subgraph-criterion} to state general subgraph criteria for the \multicut{} and \maxcut{} problem. See Figure \ref{fig:subgraph-criteria} b) for a schematic illustration.

\begin{thm}[Multicut Subgraph Criterion] \label{thm:multicut-subgraph-criterion}
Let $H = (V_H, E_H)$ be a connected subgraph of $G$ and suppose $uv \in E_H$. If
\begin{align}
\min_{y \in \mc(H)} \la \theta, y \ra = 0 \label{eq:subgraph-multicut-trivial}
\end{align}
and for all $U \subset V_H$ with $u \in U$ and $v \notin U$ it holds that
\begin{align}
\sum_{e \in \delta(U, V_H \setminus U)} \theta_e \geq \sum_{e \in \delta(V_H) \cap E^+} \theta_e, \label{eq:subgraph-multicut-condition}
\end{align}
then $x^*_{uv} = 0$ in some optimal solution $x^*$ of \eqref{eq:multicut}.
\end{thm}

In the proof of Theorem~\ref{thm:multicut-subgraph-criterion} we use the mapping $p_{V_H} \circ p_{\delta(V_H)}$ to improve solutions $x \in \mc$ with $x_f \neq 0$.
Note that the \multicut{} subgraph criterion stated in Theorem~\ref{thm:multicut-subgraph-criterion} is different from the edge and triangle criteria when evaluated on these special subgraphs.
If $H = (f,f)$ for some edge $f \in E$, then condition \eqref{eq:subgraph-multicut-condition} translates to
\begin{align*}
\theta_f \geq \sum_{e \in \delta(f) \cap E^+} \theta_e.
\end{align*}
If $H$ is a triangle, i.e.\ $H = (\{u,v,w\},\{uv,uw,vw\})$ for some vertices $u,v,w \in V$, then condition \eqref{eq:subgraph-multicut-condition} translates to
\begin{align*}
\min\{\theta_{uv} + \theta_{uw}, \theta_{uv} + \theta_{vw}, \theta_{uw} + \theta_{vw}\} \nonumber \\
\geq \sum_{e \in \delta(\{u,v,w\}) \cap E^+} \theta_e.
\end{align*}

\begin{thm}[Max-Cut Subgraph Criterion] \label{thm:max-cut-subgraph-criterion}
Let $H = (V_H, E_H)$ be a connected subgraph of $G$ and suppose $uv \in E_H$. If for all $U \subset V_H$ with $u \in U$ and $v \notin U$ it holds that
%\min_{\substack{U \subset V_H : \\ u \in U, v \notin U}}
\begin{align}
& \sum_{e \in \delta(U,V_H \setminus U)} \theta_e \nonumber \\
& \geq \min \Bigg \{ \sum_{e \in \delta(U, V \setminus V_H)} \abs{\theta_e}, \sum_{e \in \delta(V_H \setminus U, V \setminus V_H)} \abs{\theta_e} \Bigg \}, \label{eq:subgraph-max-cut-condition}
\end{align}
then $x^*_{uv} = 0$ in some optimal solution $x^*$ of \eqref{eq:max-cut}.
\end{thm}

In the proof of Theorem~\ref{thm:max-cut-subgraph-criterion} we either use the mapping $p^\triangle_{\delta(U)}$ or $p^\triangle_{\delta(V_H \setminus U)}$ to improve solutions $x \in \cut$ with $x_f \neq 0$.
Note that if $H$ is a single edge or a triangle, the subgraph criterion stated in Theorem \ref{thm:max-cut-subgraph-criterion} specializes to the edge criterion, respectively triangle criterion, where only the cuts $\delta(\{u\})$, $\delta(\{v\})$, respectively $\delta(\{u\}), \delta(\{w\}), \delta(\{v,w\})$ and $\delta(\{u,v\})$ are considered.

%% file: figures/triangle-criterion.tex
\begin{tikzpicture}[scale=1.2]
%    \draw[fill=none, densely dashed, myblue] plot[smooth cycle, tension=1] coordinates
%        {(-0.1, -0.2) (0.3, 1.3) (1.4, 0.2)};
        
    \draw[fill=none, densely dashed, myblue] plot[smooth, tension=1] coordinates
        {(-0.25,0.5) (-0.1, -0.2) (0.5,-0.3)};
        
    \draw[fill=none, densely dashed, myblue] plot[smooth, tension=1] coordinates
        {(-0.2,0.7) (0.3, 1.3) (1.0,1.05)};
        
    \draw[fill=none, densely dashed, myblue] plot[smooth, tension=1] coordinates
        {(0.7,-0.3) (1.4, 0.2) (1.25,0.8)};

	\tikzstyle{gone}=[inner sep=0pt, minimum width=0ex]      
        
	\node[vertex] (0) at (0,0) {};
	\node[vertex] (1) at (0.3,1) {};
	\node[vertex] (2) at (1.1,0.2) {};
	
	\node at (0.15,-0.15) {$u$};
	\node at (0.5,1.1) {$v$};
	\node at (1.2,0.4) {$w$};

	\draw (0) edge[thick] (1);
	\draw (1) edge[thick] (2);
	\draw (0) edge[thick] (2);
	
	\draw (0) edge[thick] (-0.3,-0.3);
	\draw (1) edge[thick] (0,1.4);
	\draw (2) edge[thick] (1.6,0.3);
	\draw (2) edge[thick] (1.4,-0.2);
	
	\draw (-0.15,0.5) edge[bend left, densely dashed, myorange] (0.5,-0.25);
	\draw (-0.1,0.7) edge[bend right, densely dashed, myorange] (0.9,1.05);
	\draw (0.7,-0.2) edge[bend left, densely dashed, myorange] (1.17,0.75);
	
\end{tikzpicture}

%% file: figures/subgraph-schematic.tex
\begin{tikzpicture}[scale=1]
    \draw[draw=black, densely dashed, myblue] plot[smooth cycle, tension=0.9] coordinates
        {(0,0) (1,1.5) (3,1.5) (2,0)};
	\draw (0.9,1.4) edge[bend left, densely dashed, myorange] (2,0);
	\node at (1,0.5) {$V_H \setminus U$};
	\node at (2.2,1.2) {$U$};
	\node at (3.4,0.2) {$V \setminus V_H$};
\end{tikzpicture}

%% file: section-algorithms.tex
% !TeX root = 0000.tex

\section{Algorithms} \label{sec:algorithms}

In this section we devise algorithms that verify, for a given instance $G = (V,E,\theta)$ of the \multicut{} or \maxcut{} problem, the persistency criteria presented in Section~\ref{sec:persistency-criteria}.

The edge and triangle criteria can be checked explicitly for all edges, respectively triangles of $G$. Note that listing all triangles of a graph can be done efficiently \cite{Schank2005}.

Therefore, we focus here on developing efficient algorithms that find subgraphs $H$ which qualify for the criteria from Theorem \ref{thm:multicut-subgraph-criterion} and \ref{thm:max-cut-subgraph-criterion}. 
Specifically, we propose routines that (i) check for a given connected subgraph $H$ whether some persistency criteria apply and (ii) find good candidates for $H$.

\subsection{Subgraph evaluation}

Let $H = (V_H, E_H)$ be a subgraph of $G$ that we want to check for persistency condition \eqref{eq:subgraph-multicut-condition}, respectively \eqref{eq:subgraph-max-cut-condition}.
Now, for a given edge $uv \in E_H$, we can determine if~\eqref{eq:subgraph-multicut-condition} holds true for all $U \subset V_H$ with $u \in U$ and $v \notin U$ by minimizing the left-hand side w.r.t.\ $U$.
In contrast, for \eqref{eq:subgraph-max-cut-condition}, we also need to simultaneously maximize the right-hand side, since it depends on $U$ as well.
Obviously, minimizing the left-hand side (of either \eqref{eq:subgraph-multicut-condition} or \eqref{eq:subgraph-max-cut-condition}) means finding a minimum $u$-$v$-cut w.r.t.\ $\theta$. 
Further, since the right-hand sides are non-negative, the minimum $u$-$v$-cut must have non-negative weight. However, in general the weights $\theta$ on $H$ may be negative, which renders both optimization problems hard in general.

For this reason, we simplify the problem by restriction to suitable subgraphs $H$ that satisfy Assumption \ref{assumption} below. In order to state Assumption \ref{assumption} rigorously, we need to briefly recap the following integer linear programming (ILP) formulation of the \multicut{} problem.

The \multicut{} problem can be stated equivalently to~\eqref{eq:multicut} as finding the minimum weight edge set w.r.t.\ $\abs{\theta}$ that covers every cycle with exactly one negative edge, the so-called \emph{erroneous} or \emph{conflicted} cycles \cite{Demaine2006, Lange2018}. The corresponding ILP formulation reads
\begin{align}
\min_{\hat x} \; & \la \abs{\theta}, \hat x \ra + \sum_{e \in E^-} \theta_e \label{eq:covering-primal} \\
\text{s.t.} \; & \sum_{e \in C} \hat x_e  \geq 1, \quad \forall \text{ conflicted } C \label{eq:covering-constraint} \\
& \hat x \in \{0,1\}^E. \nonumber
\end{align}
The associated packing dual is the linear program
\begin{align}
\max_\lambda \; & \la \1, \lambda \ra + \sum_{e \in E^-} \theta_e \label{eq:packing-dual} \\
\text{ s.t.} \; & \sum_{C : e \in C} \lambda_C \leq \abs{\theta_e} & \forall e \in E, \nonumber \\
& \lambda \geq 0. \nonumber
\end{align}
For any dual feasible $\lambda \geq 0$, the associated reduced costs for the primal problem \eqref{eq:covering-primal} are given by
\begin{align}
\tilde \theta_e := \bigg ( \abs{\theta_e} - \sum_{C : e \in C} \lambda_C \bigg ) \sign \theta_e \qquad \forall e \in E.\label{eq:reduced-cost}
\end{align}

\begin{assump}
\label{assumption}
Let $H = (V_H,E_H,\theta)$ be a weighted graph such that
\begin{enumerate}[i)]
\item The graph $H$ has a trivial \multicut{} solution, i.e.\ $\min_{y \in \mc(H)} \la \theta, y \ra = 0$. \label{assumption-trivial-solution}
\item An optimal packing dual solution $\lambda^*$ for the \multicut{} problem on $H$ is available. \label{assumption-optimal-dual-solution}
\end{enumerate}
\end{assump}

Note that Assumption \ref{assumption} \ref{assumption-trivial-solution} also implies a trivial \maxcut{} solution, since $\cut \subseteq \mc$.
Assumption \ref{assumption} has the following expedient consequence.

\begin{lem} \label{lem:lower-bound}
Let $H=(V_H,E_H,\theta)$ be a weighted graph that satisfies Assumption \ref{assumption}. Then $\tilde \theta_e \geq 0$ for all $e \in E_H$ and for any cut $\delta(U)$ of $H$ it holds that
\begin{align*}
0 \leq \sum_{e \in \delta(U)} \tilde \theta_e \leq \sum_{e \in \delta(U)} \theta_e.
\end{align*}
\end{lem}

Our method exploits Assumption \ref{assumption} and Lemma \ref{lem:lower-bound} as follows. First, we compute a solution to the packing dual~\eqref{eq:packing-dual} by the fast heuristic \emph{Iterative Cycle Packing} (ICP) algorithm from \cite{Lange2018}.
Then, if the computed dual solution shows that the \multicut{} solution on $H$ is trivial, we can compute lower bounds to the left-hand side of~\eqref{eq:subgraph-multicut-condition} and \eqref{eq:subgraph-max-cut-condition} by applying max-flow techniques on $H$ with capacities $\tilde \theta$.

In the case of the \multicut{} problem, the right-hand side of~\eqref{eq:subgraph-multicut-condition} is constant w.r.t.\ $U$ so it suffices to compute a Gomory-Hu tree on $H$.
In the case of the \maxcut{} problem, however, this is not sufficient, since the right-hand side of~\eqref{eq:subgraph-max-cut-condition} also depends on $U$. 
Here, after replacing $\theta_e$ by $\tilde \theta_e$ for all $e \in E_H$, we need to solve the following $\min \max$ problem
\begin{align}
& \min_{\substack{U \subset V_H : \\ u \in U, v \notin U}} \Bigg ( \sum_{e \in \delta(U, V_H \setminus U)} \tilde \theta_e \nonumber \\
& \quad - \min \Bigg \{ \sum_{e \in \delta(U, V \setminus V_H)} \abs{\theta_e}, \sum_{e \in \delta(V_H \setminus U, V \setminus V_H)} \abs{\theta_e} \Bigg \} \Bigg ) \nonumber \\
& = - \sum_{e \in \delta(V_H)} \abs{\theta_e} + \min_{\substack{U \subset V_H : \\ u \in U, v \notin U}} \Bigg ( \sum_{e \in \delta(U, V_H \setminus U)} \tilde \theta_e \nonumber \\
& \quad + \max \Bigg \{ \sum_{e \in \delta(U, V \setminus V_H)} \abs{\theta_e}, \sum_{e \in \delta(V_H \setminus U, V \setminus V_H)} \abs{\theta_e} \Bigg \} \Bigg ). \label{eq:min-max-problem}
\end{align}

As solving this problem exactly appears to be difficult, we propose to solve a relaxation that is obtained by replacing the inner $\max$ term with
\begin{align*}
\max_{\alpha \in [0,1]} \; \alpha \sum_{e \in \delta(U, V \setminus V_H)} \abs{\theta_e} + (1-\alpha) \sum_{e \in \delta(V_H \setminus U, V \setminus V_H)} \abs{\theta_e}
\end{align*}
and then swapping the order of $\min$ and $\max$. This yields
\begin{align}
& \eqref{eq:min-max-problem} \geq \nonumber \\
& - \sum_{e \in \delta(V_H)} \abs{\theta_e} + \max_{\alpha \in [0,1]} \min_{\substack{U \subset V_H : \\ u \in U, v \notin U}} \Bigg ( \sum_{e \in \delta(U, V_H \setminus U)} \tilde \theta_e + \nonumber \\
& \quad \alpha \sum_{e \in \delta(U, V \setminus V_H)} \abs{\theta_e} + (1-\alpha) \sum_{e \in \delta(V_H \setminus U, V \setminus V_H)} \abs{\theta_e} \Bigg ). \label{eq:max-min-relaxation}
\end{align}
The right-hand side is the maximization of a concave, non-smooth function on the unit interval, which can be solved efficiently with the bisection method. In every iteration, the inner minimization problem needs to be solved for a fixed $\alpha \in [0,1]$,  which can be formulated again as a max-flow problem.

For solving the max-flow problems that occur in our method, we use Boykov-Kolmogorov's algorithm with reused search trees \cite{Boykov2004,Kohli2005}. For computing Gomory-Hu trees, we use a parallelized implementation of Gusfield's algorithm~\cite{Gusfield1990,Cohen2011}.

\subsection{Finding candidate subgraphs}

To efficiently find good candidate subgraphs, we employ the following strategy. First, we compute a primal feasible solution $\bar x \in X$ by a fast heuristic method such as greedy edge contraction algorithms \cite{Kahruman2007, Keuper2015b}. If the heuristic solution $\bar x$ is reasonably good, then many components defined by $\bar x$ should be close to optimal. Thus, in the case of the \multicut{} problem, the components may already serve as candidate subgraphs. In the case of the \maxcut{} problem, we use $\bar x$ to transform the instance by the switching operation described in Section \ref{sec:improving-mappings}.

Then, we compute a heuristic packing dual solution $\bar \lambda$ by ICP for the entire graph $G = (V,E,\theta)$. The candidate subgraphs are determined as the connected components of the positive residual graph $(V, \{e \in E \mid \tilde \theta_e > 0 \})$, where $\tilde \theta$ is defined similarly as before in \eqref{eq:reduced-cost}. The intuition behind this strategy is that, by construction, the subgraphs' inner edges have relatively higher weight than their outgoing edges. This facilitates the application of the conditions \eqref{eq:subgraph-multicut-condition} and \eqref{eq:subgraph-max-cut-condition}.

\myparagraph{Reduced cost fixing.}
Further, whenever both a primal solution and dual solution are available, we use the following technique known as \emph{reduced cost fixing}~\cite{Balas1980} to determine additional persistent variables.
 Let $\gamma := \la \theta, \bar x \ra - \la \1, \bar \lambda \ra - \sum_{e \in E^-} \theta_e$ denote the duality gap of the primal-dual solution pair and suppose $\gamma < \tilde \theta_f$ for some $f \in E$. Then, it follows that $x_f = 1$ cannot be optimal and thus we can fix $x_f = 0$.

%% file: section-experiments.tex
% !TeX root = 0000.tex

\begin{table}[!t]
	\caption{The table gathers for each data set the number of instances ($\#I$), the graph sizes and instance type (\ref{eq:problem}).}
	\label{tab:datasets}
	\small
	\setlength{\tabcolsep}{4pt}
	\centering
	\begin{tabular}{lrccc}
		\toprule
		Data set & $\#I$ & $\lvert V \rvert$ & $\lvert E \rvert$ & \ref{eq:problem} \\
		\midrule
		\emph{Image Seg.} & 100 & 156--3764 & 439--10970 & \ref{eq:multicut} \\
		\emph{Knott-3D-150} & 8 & 572--972 & 3381--5656 & \ref{eq:multicut} \\
		\emph{Knott-3D-300} & 8 & 3846--5896 & 23k--36k & \ref{eq:multicut} \\
		\emph{Knott-3D-450} & 8 & 15k--17k & 94k--107k & \ref{eq:multicut} \\
		\emph{Knott-3D-550} & 8 & 27k--31k & 173k--195k & \ref{eq:multicut} \\
		\emph{Mod. Clustering} & 6 & 34--115 & 561--6555 & \ref{eq:multicut} \\
		\midrule
		\emph{CREMI-small} & 3 & 20k--35k & 170k--235k & \ref{eq:multicut} \\
		\emph{CREMI-large} & 3 & 430k--620k & 3.2m--4.1m & \ref{eq:multicut} \\
		\midrule
		\emph{Fruit-Fly Level 1--4} & 4 & 5m--11m & 28m--72m & \ref{eq:multicut} \\
		\emph{Fruit-Fly Global} & 1 & 90m & 650m & \ref{eq:multicut} \\
		\midrule
		\textit{Ising Chain} & 30 & 100--300 & 4950--44850 & \ref{eq:max-cut} \\
		\textit{2D Torus} & 9 & 100--400 & 200--800 & \ref{eq:max-cut} \\
		\textit{3D Torus} & 9 & 125--343 & 375--1029 & \ref{eq:max-cut} \\
		\midrule
		\textit{Deconvolution} & 2 & 1001 & 11k--34k & \ref{eq:max-cut} \\
		\textit{Super Resolution} & 2 & 5247 & 15k--25k & \ref{eq:max-cut} \\
		\textit{Texture Restoration} & 4 & 7k--22k & 59k--195k & \ref{eq:max-cut} \\
		\bottomrule
	\end{tabular}
\end{table}

\section{Experiments} \label{sec:experiments}

In order to study the effectiveness of our methods, we evaluate them on a collection of more than 200 instances from the literature. The size of the instances ranges from a few hundred to hundreds of millions of variables (edges). As a performance measure, we use the average relative size reduction of test instances that is obtained by applying our algorithms.

\begin{table}
	\caption{For each dataset the table reports the average fraction of remaining nodes and edges after applying our method, respectively the method from \cite{Lange2018} (lower is better). $^\dagger$Results for \emph{Fruit-Fly Global} are without ICP-based candidate subgraphs.}
	\label{tab:shrinkage-multicut}
	\small
	\setlength{\tabcolsep}{7pt}
	\centering
	\begin{tabular}{lrrrr}
		\toprule
		& \multicolumn{2}{c}{Our} & \multicolumn{2}{c}{\cite{Lange2018}}  \\
		\cmidrule(lr){2-3} \cmidrule(lr){4-5}
		Data set & $\lvert V \rvert$ & $\lvert E \rvert$ & $\lvert V \rvert$ & $\lvert E \rvert$ \\
		\midrule
		%\emph{Image Seg.} & $46.0\%$ & $45.3\%$ & $63.7\%$ & $62.7\%$ \\
		\emph{Image Seg.} & $27.7\%$ & $27.4\%$ & $63.7\%$ & $62.7\%$ \\
		%\emph{Knott-3D-150} & $28.9\%$ & $29.6\%$ & $75.2\%$ & $88.3\%$ \\
		\emph{Knott-3D-150} & $9.7\%$ & $9.6\%$ & $75.2\%$ & $88.3\%$ \\
		%\emph{Knott-3D-300} & $61.8\%$ & $70.2\%$ & $76.7\%$ & $91.6\%$ \\
		\emph{Knott-3D-300} & $54.8\%$ & $61.6\%$ & $76.7\%$ & $91.6\%$ \\
		%\emph{Knott-3D-450} & $67.8\%$ & $78.7\%$ & $77.6\%$ & $92.4\%$ \\
		\emph{Knott-3D-450} & $66.9\%$ & $77.6\%$ & $77.6\%$ & $92.4\%$ \\
		%\emph{Knott-3D-550} & $69.4\%$ & $81.0\%$ & $77.8\%$ & $92.6\%$ \\
		\emph{Knott-3D-550} & $67.8\%$ & $79.0\%$ & $77.8\%$ & $92.6\%$ \\
		%\emph{Mod. Clustering} & $91.3\%$ & $84.1\%$ & $92.0\%$ & $85.1\%$ \\
		\emph{Mod. Clustering} & $88.7\%$ & $80.6\%$ & $92.0\%$ & $85.1\%$ \\
		\midrule
		%\emph{CREMI-small} & $37.8\%$ & $35.6\%$ & $76.6\%$ & $75.3\%$ \\
		\emph{CREMI-small} & $33.8\%$ & $31.9\%$ & $76.6\%$ & $75.3\%$ \\
		%\emph{CREMI-large} & $46.6\%$ & $46.9\%$ & $83.7\%$ & $86.6\%$ \\
		\emph{CREMI-large} & $44.0\%$ & $44.2\%$ & $83.7\%$ & $86.6\%$ \\
		\midrule
		%\emph{Fruit-Fly Level 1--4} & $9.4\%$ & $10.1\%$ & $24.6\%$ & $27.9\%$ \\
		\emph{Fruit-Fly Level 1--4} & $8.7\%$ & $9.6\%$ & $24.6\%$ & $27.9\%$ \\
		\emph{Fruit-Fly Global}$^\dagger$ & $56.3\%$ & $51.8\%$ & $77.9\%$ & $74.5\%$ \\
		\bottomrule
	\end{tabular}
\end{table}

\myparagraph{Instances.}
For the \multicut{} problem we use segmentation and clustering instances from the OpenGM benchmark~\cite{Kappes2015} as well as biomedical segmentation instances provided by the authors of~\cite{Beier2017} and \cite{Pape2017}.
The dataset \emph{Image Segmentation} contains planar graphs that are constructed from superpixel adjacencies of photographs.
The \emph{Knott-3D} data sets contains non-planar graph arising from volume images acquired by electron microscopy.
The set \emph{Modularity Clustering} contains complete graphs constructed from clustering problems on small social networks.
The \emph{CREMI} data sets contain supervoxel adjacency graphs obtained from volume image scans of neural tissue.
The \emph{Fruit-Fly} instances were generated from volume image scans of fruit fly brain matter. The global problem is the largest instance in this study with roughly 650 million variables. It represents the current limit of what can be tackled by state-of-the-art local search algorithms. The instances \emph{Level 1--4} are progressively simplified versions of the global problem obtained via block-wise domain decomposition \cite{Pape2017}.

For the \maxcut{} problem we use two different types of instances.
(i)~The datasets \emph{Ising Chain}, \emph{2D Torus} and \emph{3D Torus} contain instances that stem from applications in statistical physics \cite{Liers2005}.
The instances in \emph{Ising Chain} assume a linear order on the nodes.
For any pair of nodes there is an edge with an associated weight.
The absolute values of the weights decrease exponentially with the distance of the nodes in the linear order.
The instances in \emph{2D Torus} and \emph{3D Torus} are defined on toroidal grid graphs in two, resp.\ three dimensions with Gaussian distributed weights.
(ii)~The datasets \emph{Deconvolution}, \emph{Super Resolution} and \emph{Texture Restoration} contain QPBO instances originating from image processing applications~\cite{Rother2007, Verma2012} that are converted to our formulation of the \maxcut{} problem.
The transformation introduces an additional node that is connected with all other nodes. 
A cut (uncut) edge to the additional node signifies label 0 (resp.\ 1).
The instance size statistics for all data sets are summarized in Table \ref{tab:datasets}.

\begin{figure}
\input{figures/ablation-multicut.tex}
\vspace{1ex}
\caption{The figure shows the average fraction of remaining variables after shrinking the instance with progressively more expensive persistency criteria. The criteria added are from left to right: none [$\varnothing$], connected components of $G^+$~\cite{Alush2012}, single node cuts \cite{Lange2018}, edge subgraphs [Edge], triangle subgraphs [$\triangle$], greedy subgraphs [Greedy], ICP candidate subgraphs and reduced cost fixing [ICP].}
\label{fig:ablation-multicut}
\end{figure}
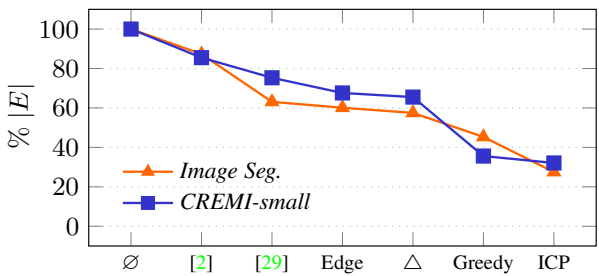

\myparagraph{Results.}
In Table \ref{tab:shrinkage-multicut}, we report for the \multicut{} instances the average graph sizes after shrinking the instances with our algorithms from Section \ref{sec:algorithms}.
In Figure~\ref{fig:ablation-multicut}, the contributions of the individual persistency criteria are separated and compared to prior work.
It can be seen from Table \ref{tab:shrinkage-multicut} and Figure~\ref{fig:ablation-multicut} that our criteria enable finding substantially more persistent variables than the prior work \cite{Alush2012, Lange2018}.
In relation to the graph sizes after shrinking with the baseline \cite{Lange2018}, our method achieves an additional size reduction of about 30--60\% for the large \emph{CREMI} and \emph{Fruit-Fly} instances.
This shows that our algorithms find persistent variable assignments that are harder to detect than with the criteria from prior work.

\begin{table}[!t]
	\caption{For each dataset, the table reports the average fraction of remaining nodes and edges after applying our method, respectively the QPBO method \cite{Rother2007} (lower is better). Note that the latter is not applicable to original \maxcut{} instances due to symmetries.}
	\label{tab:shrinkage-max-cut}
	\small
	\setlength{\tabcolsep}{7pt}
	\begin{tabular}{lrrrrr}
		\toprule
		& \multicolumn{2}{c}{Our} & \multicolumn{2}{c}{\cite{Rother2007}}  \\
		\cmidrule(lr){2-3} \cmidrule(lr){4-5}
		Data set & $\lvert V \rvert$ & $\lvert E \rvert$ & $\lvert V \rvert$ & $\lvert E \rvert$ \\
		\midrule
		\textit{Ising Chain} & $0.0\%$ & $0.0\%$ & n/a & n/a \\
		\textit{2D Torus} & $23.6\%$ & $27.9\%$ & n/a & n/a \\
		\textit{3D Torus} & $94.8\%$ & $98.1\%$ & n/a & n/a \\
		\midrule
		\textit{Deconvolution} & $61.0\%$ & $56.5\%$ & $61.0\%$ & $56.5\%$ \\
		\textit{Super Resolution} & $0.0\%$ & $0.0\%$ & $0.2\%$ & $0.1\%$ \\
		\textit{Texture Restoration} & $98.4\%$ & $98.5\%$ & $58.8\%$ & $57.3\%$ \\
		\bottomrule
	\end{tabular}
\end{table}

In Table \ref{tab:shrinkage-max-cut} we report for the \maxcut{} instances the average graph size reduction on each dataset.
For the QPBO instances we compare to the QPBO method \cite{Rother2007}.
For the original \maxcut{} instances we are unaware of any baseline method and the QPBO method is not applicable.
In Figure~\ref{fig:ablation-max-cut} we compare the contribution of the different subgraph criteria.
It can be seen that our method solves all \emph{Ising Chain} instances to optimality, which is facilitated by their particular distribution of the weights. On \emph{2D Torus} we achieve substantial size reductions and on the denser \emph{3D Torus} instances we find few persistencies.
Our results on the QPBO instances are on a par with~\cite{Rother2007} for \emph{Deconvolution} and \emph{Super Resolution} while our method is less effective for \emph{Texture Restoration}.

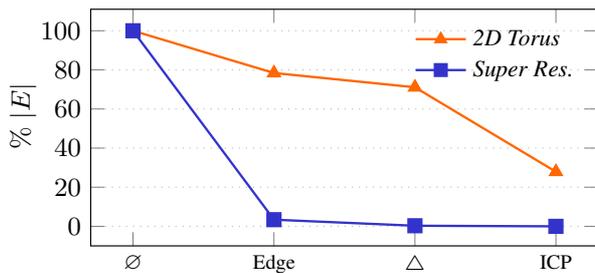
\begin{figure}[t!]
\input{figures/ablation-max-cut.tex}
\vspace{1ex}
\caption{The figure shows the average fraction of remaining variables after shrinking the instance with progressively more expensive persistency criteria. The criteria added are from left to right: none [$\varnothing$], edge subgraphs [Edge], triangle subgraphs [$\triangle$], ICP candidate subgraphs and reduced cost fixing [ICP].}
\label{fig:ablation-max-cut}
\end{figure}

%% file: figures/ablation-multicut.tex
\begin{tikzpicture}
%\pgfplotsset{every x tick label/.append style={font=\scriptsize}}
\small
\begin{axis}[
	width=0.85\columnwidth,
	height=0.5\columnwidth,
	ylabel style={yshift=-1.5em},
	ymin=-10,
	ymax=110,
	ytick={100,80,60,40,20,0},
	grid style=dotted,
	ymajorgrids=true,
	%xmode=log,
	%title={\multicut},
	%xlabel={Criteria added},
	xtick={0,...,6},
	xticklabel style = {font=\scriptsize},
	xticklabels={\footnotesize $\varnothing$, \cite{Alush2012}, \cite{Lange2018}, Edge, $\triangle$, Greedy, ICP},
	ylabel={$\% \;\abs{E}$},
	legend entries={\emph{Image Seg.}, \emph{CREMI-small}, \emph{CREMI-large}},
	legend style={
		fill=none,
		draw=none,
		font=\footnotesize,
		at={(0.05, 0.08)},
		anchor=south west,
	},
	legend cell align=left
]
\addplot[color=myorange, mark=triangle*, thick] table [col sep=space, x index=0, y index=1] {data/image-seg.txt};
\addplot[color=myblue, mark=square*, thick] table [col sep=space, x index=0, y index=1] {data/cremi-small.txt};
%\addplot[color=mygreen, mark=*, thick] table [col sep=space, x index=0, y index=1] {data/cremi-small.txt};
\end{axis}
\end{tikzpicture}

%% file: figures/ablation-max-cut.tex
\begin{tikzpicture}
%\pgfplotsset{every x tick label/.append style={font=\scriptsize}}
\small
\begin{axis}[
	width=0.85\columnwidth,
	height=0.5\columnwidth,
	ylabel style={yshift=-1.5em},
	ymin=-10,
	ytick={100,80,60,40,20,0},
	grid style=dotted,
	ymajorgrids=true,
	%xmode=log,
	%title={\maxcut},
	%xlabel={Criteria added},
	xtick={0,...,3},
	xticklabel style = {font=\scriptsize},
	xticklabels={\footnotesize $\varnothing$, Edge, $\triangle$, ICP},
	ylabel={$\% \;\abs{E}$},
	legend entries={\emph{2D Torus}, \emph{Super Res.}},
	legend style={
		fill=none,
		draw=none,
		font=\footnotesize,
		at={(0.9, 0.88)},
		anchor=north east,
	},
	legend cell align=left
]
\addplot[color=myorange, mark=triangle*, thick] table [col sep=space, x index=0, y index=1] {data/2d-torus.txt};
\addplot[color=myblue, mark=square*, thick] table [col sep=space, x index=0, y index=1] {data/super-res.txt};
%\addplot[color=mygreen, mark=*, thick] table [col sep=space, x index=0, y index=1] {data/cremi-small.txt};
\end{axis}
\end{tikzpicture}

%% file: appendix.tex
\appendix

\section{Proofs}

\paragraph{Proof of Lemma \ref{lem:well-definedness}.}

\begin{proof}
(i)
Let $x \in \mc$ and assume that $z = p_{\delta(U)}(x) \notin \mc$. Then there exists a cycle $C$ with exactly one cut edge in $z$, i.e.\ $z_f = 1$ for some $f \in C$ and $z_e = 0$ for all $e \in C \setminus f$. It holds that $x_f = 1$ and thus $C$ crosses $\delta(U)$ exactly once, which is impossible.

(ii)
Let $x \in \mc$ and assume that $z = p_U(x) \notin \mc$.
Then there is a cycle $C$ with $z_f = 1$ for some $f \in C$ and $z_e = 0$ for all $e \in C \setminus f$.
Since $z \leq x$ there exists an edge $uv = g \in C$, $g \neq f$ with $x_{g} = 1$ and $z_{g}= 0$.
Then, according to \eqref{eq:elementary-join-mapping-definition}, there exists a $uv$-path $P$ such that $x_e = 0$ for all $e \in P \setminus E(U)$.
Replace the cycle $C$ by $C \triangle (P \cup \{g\})$.
Repeating this argument for all such edges $g \in C$ yields a path $P'$ connecting the endpoints of $f$ such that $x_{e} = 0$ for all $e \in P' \setminus E(U)$, which is a contradiction to $z_f = 1$.
\end{proof}

\paragraph{Proof of Theorem \ref{thm:single-edge-criterion}.}

\begin{proof}
First, we show that the mapping
\begin{align*}
p(x) = \begin{cases} p^\triangle_{\delta(U)}(x) & \text{if } x_f \neq \beta \\ x & \text{else} \end{cases}
\end{align*}
is improving for the \maxcut{} problem. Let $x \in \cut$, $z = p(x)$ and suppose  $x_f \neq \beta$. It holds that
\begin{align*}
\la \theta, z \ra - \la \theta, x \ra & = \theta_f (z_f-x_f) + \sum_{e \in \delta(U) \setminus \{f\}} \theta_e (z_e-x_e) \\
& = - \abs{\theta_f} + \sum_{ e \in \delta(U) \setminus \{f\}} \theta_e (z_e-x_e) \\
& \leq - \abs{\theta_f} + \sum_{ e \in \delta(U) \setminus \{f\}} \abs{\theta_e} \\
& \leq 0.
\end{align*}
Similarly, for $\beta = 0$, we show that the mapping
\begin{align*}
p(x) = \begin{cases} (p_{f} \circ p_{\delta(U)})(x) & \text{if } x_f \neq \beta \\ x & \text{else} \end{cases}
\end{align*}
is improving for the \multicut{} problem. Let $x \in \mc$, $z = p(x)$ and suppose $x_f \neq \beta$. It holds that
\begin{align*}
\la \theta, z \ra - \la \theta, x \ra & = \theta_f (0-1) + \sum_{e \in \delta(U) \setminus \{f\}} \theta_e (z_e-x_e) \\
& \leq - \abs{\theta_f} + \sum_{ e \in \delta(U) \setminus \{f\}} \abs{\theta_e} \\
& \leq 0.
\end{align*}
Finally, for $\beta = 1$, we show that the mapping
\begin{align*}
p(x) = \begin{cases} p_{\delta(U)}(x) & \text{if } x_f \neq \beta \\ x & \text{else} \end{cases}
\end{align*}
is improving for multicut. Let $x \in \mc$, $z = p(x)$ and suppose $x_f \neq \beta$. It holds that
\begin{align*}
\la \theta, z \ra - \la \theta, x \ra & = \theta_f (1-0) + \sum_{e \in \delta(U) \setminus \{f\}} \theta_e (1-x_e) \\
& \leq -\abs{\theta_f} + \sum_{ e \in \delta(U) \cap E^+} \theta_e (1-x_e) \\
& \leq - \abs{\theta_f} + \sum_{ e \in \delta(U) \cap E^+} \abs{\theta_e} \\
& \leq 0.
\end{align*}
This concludes the proof.
\end{proof}

\paragraph{Proof of Corollary \ref{cor:triangle-criterion}.}

\begin{proof}
We use Lemma~\ref{lemma:subgraph-criterion}:
\begin{enumerate}[wide, labelwidth=!, labelindent=0pt]
\item[(i)] 
In the case $x_{uw} = 1$, $x_{uv} = 1$, $x_{vw} = 0$ apply $p^\triangle_{\delta(U)}$.
In the case $x_{uw} = 1$, $x_{uv} = 0$, $x_{vw} = 1$ apply $p^\triangle_{\delta(W)}$. These mappings are improving due to \eqref{eq:triangle-criterion-1} and \eqref{eq:triangle-criterion-2}.
\item[(ii)]
In the case $x_{uw} = 1$, $x_{uv} = 1$, $x_{vw} = 0$ apply $p_{\{u,w\}} \circ p_{\delta(U)}$.
In the case $x_{uw} = 1$, $x_{uv} = 0$, $x_{vw} = 1$ apply $p_{\{u,w\}} \circ p_{\delta(W)}$. These mappings are improving analogously to (i).
In the additional case $x_{uw} = 1$, $x_{uv} = 1$, $x_{vw} = 1$ apply the mapping $p = p_{\{u,v,w\}} \circ p_{\delta(\{u,v,w\})}$. It is improving, since
\begin{align*}
& \la \theta, p(x) \ra - \la \theta, x \ra \\
& = \sum_{e \in \delta(\{u,v,w\})} \theta_e(1 - x_e) - \theta_{uv} - \theta_{uw} - \theta_{vw} \\
& \leq \sum_{e \in \delta(\{u,v,w\}) \cap E^+} \theta_e(1 - x_e) - \sum_{e \in \delta(\{u,v,w\}) \cap E^+} \theta_e \\
& \leq 0. \qedhere
\end{align*}
\end{enumerate}
\end{proof}

\paragraph{Proof of Theorem \ref{thm:multicut-subgraph-criterion}.}

\begin{proof}
We use Lemma \ref{lemma:subgraph-criterion}. Let $y \in \mc(H)$ with $y_{uv} = 1$ and suppose $x \in \mc$ with $x_{|E_H} = y$. Then there is a multicut $M$ of $H$ such that $y = \1_M$. Due to \eqref{eq:subgraph-multicut-trivial}, every (multi-)cut of $H$ has nonnegative weight. Therefore, there exists some $U \subset V_H$ with $u \in U$ and $v \notin U$ such that $\delta(U, V_H \setminus U) \subseteq M$ and
\begin{align}
\sum_{e \in E_H} \theta_e x_e = \sum_{e \in M} \theta_e \geq \sum_{e \in \delta(U, V_H \setminus U)} \theta_e. \label{eq:subgraph-multicut-partial}
\end{align}
Let $p^y(x) = (p_{V_H} \circ p_{\delta(V_H)})(x)$, then it follows from \eqref{eq:subgraph-multicut-condition} and \eqref{eq:subgraph-multicut-partial} that
\begin{align*}
& \la \theta, p^y(x) \ra - \la \theta, x \ra \\
& = \sum_{e \in \delta(V_H)} \theta_e(1 - x_e) - \sum_{e \in E_H} \theta_e x_e \\
& \leq \sum_{e \in \delta(V_H) \cap E^+} \theta_e(1-x_e) - \sum_{e \in \delta(V_H) \cap E^+} \theta_e \\
& \leq 0. \qedhere
\end{align*}
\end{proof}

\paragraph{Proof of Theorem \ref{thm:max-cut-subgraph-criterion}.}

\begin{proof}
We use Lemma \ref{lemma:subgraph-criterion}. Suppose $y \in \cut(H)$ with $y_{uv} = 1$. Let $U \subset V_H$ be such that $y$ is the incidence vector of $\delta(U, V_H \setminus U)$ in $H$ and suppose $x \in \cut$ with $x_{|E_H} = y$. We may assume that
\begin{align*}
\sum_{e \in \delta(U,V_H \setminus U)} \theta_e \geq \sum_{e \in \delta(U, V \setminus V_H)} \abs{\theta_e},
\end{align*}
otherwise redefine $U := V_H \setminus U$. Now, let $z = p^y(x) = p^\triangle_{\delta(U)}(x)$, then it follows that
\begin{align*}
& \la \theta, z \ra - \la \theta, x \ra = \sum_{e \in \delta(U)} \theta_e (z_e - x_e) \\
& = \sum_{e \in \delta(U, V_H \setminus U)} \theta_e(0 - 1) + \sum_{e \in \delta(U, V \setminus V_H)} \theta_e(z_e - x_e) \\
& \leq \sum_{e \in \delta(U, V_H \setminus U)} -\theta_e + \sum_{e \in \delta(U, V \setminus V_H)} \abs{\theta_e} \\
& \leq 0. \qedhere
\end{align*}
\end{proof}

\paragraph{Proof of Lemma \ref{lem:lower-bound}.}

\begin{proof}
As $H$ satisfies Assumption \ref{assumption}, the dual problem~\eqref{eq:packing-dual} evaluates to zero. Thus, since any conflicted cycle contains exactly one edge $e \in E_H$ with $\theta_e < 0$, we must have $\sum_{C : e \in C} \lambda^*_C = \abs{\theta_e}$, which implies $\tilde \theta_e = 0$.
Furthermore, for any cut $\delta(U)$ of $H$ it holds that
\begin{align*}
& \sum_{e \in \delta(U)} \tilde \theta_e \\
& = \sum_{e \in \delta(U) \cap E^+} \tilde \theta_e + \sum_{e \in \delta(U) \cap E^-} \tilde \theta_e \\
& = \sum_{e \in \delta(U) \cap E^+} \bigg ( \theta_e - \sum_{C : e \in C} \lambda^*_C \bigg ) + \\
& \quad \sum_{e \in \delta(U) \cap E^-} \bigg ( \theta_e + \sum_{C : e \in C} \lambda^*_C \bigg ) \\
& = \sum_{e \in \delta(U)} \theta_e + \sum_{\substack{e \in \delta(U) \cap E^- \\ C : e \in C}} \lambda^*_C - \sum_{\substack{e \in \delta(U) \cap E^+ \\ C : e \in C}} \lambda^*_C \\
& \leq \sum_{e \in \delta(U)} \theta_e.
\end{align*}
The last inequality holds true, because every cycle with precisely one negative edge $e$, where $e \in \delta(U) \cap E^-$, also contains some positive edge $f \in \delta(U) \cap E^+$, as it crosses $\delta(U)$ at least twice. This concludes the proof.
\end{proof}

\section{Persistency criteria}

In this section we describe a technical improvement of the \multicut{} subgraph criterion presented in Theorem~\ref{thm:multicut-subgraph-criterion}. Here, improvement means relaxing the inequality \eqref{eq:subgraph-multicut-condition} such that it applies more often (without compromising the persistency result).

To this end, we need to introduce some more notation. For any set of vertices $U \subseteq V$, let 
\begin{align}
\partial U := \{v \in V \mid \exists uv \in E \text{ with } u \in U \}
\end{align}
denote the \emph{boundary} of $U$ in $V$. The boundary of $U$ consists of those vertices in $V$ that have a neighbor in $U$ but are not in $U$ themselves. For any set $U \subseteq V$, its \emph{closure} $\closure U$ is defined as the union of $U$ with its boundary, i.e.\
\begin{align*}
\closure U := U \cup \partial U.
\end{align*}
Further, for any connected subgraph $H = (V_H, E_H)$, we define its \emph{positive closure} as the subgraph
\begin{align*}
\closure H := (\closure{V_H}, E_H \cup (\delta(V_H) \cap E^+) ),
\end{align*}
which additionally includes all positive edges between $V_H$ and its boundary.

Below we state a more refined version of Theorem~\ref{thm:multicut-subgraph-criterion}. The difference in Theorem~\ref{thm:multicut-criterion-with-boundary} is that the inner cut is w.r.t.\ the subgraph $\closure H$ instead of $H$.
\begin{thm}[Multicut Subgraph Criterion] \label{thm:multicut-criterion-with-boundary}
Let $H = (V_H, E_H)$ be a connected subgraph of $G$ and suppose $uv \in E_H$. If
\begin{align*}
\min_{y \in \mc(H)} \la \theta, y \ra = 0
\end{align*}
and for all $U \subset \closure{V_H}$ with $u \in U$ and $v \notin U$ it holds that
\begin{align}
\sum_{e \in \delta(U, \closure{V_H} \setminus U) \cap E_{\closure H}} \theta_e \geq \sum_{e \in \delta(V_H) \cap E^+} \theta_e, \label{eq:subgraph-multicut-condition-with-boundary}
\end{align}
then $x^*_{uv} = 0$ in some optimal solution $x^*$ of \eqref{eq:multicut}.
\end{thm}

\begin{proof}
The proof is largely analogous to the proof of Theorem \ref{thm:multicut-subgraph-criterion}. Suppose $x \in \mc$ with $x_{|E_H} = y \in \mc(H)$ and $y_{uv} = 1$. Apparently, there exists a multicut $M$ of $\closure H$ which extends $y$ such that $x_{|E_{\closure H}} = \1_M$.
Similarly to before, there exists some $U \subset \closure{V_H}$ with $u \in U$ and $v \notin U$ such that $\delta(U, \closure{V_H} \setminus U) \cap E_{\closure H} \subseteq M$ and
\begin{align}
\sum_{e \in \delta(V_H) \cap E^+} \theta_e x_e + \sum_{e \in E_H} \theta_e x_e \geq \sum_{e \in \delta(U, \closure{V_H} \setminus U) \cap E_{\closure H}} \theta_e. \label{eq:subgraph-multicut-partial-with-boundary}
\end{align}
Eventually, using \eqref{eq:subgraph-multicut-condition-with-boundary} and \eqref{eq:subgraph-multicut-partial-with-boundary}, we show that the mapping $p^y = (p_{V_H} \circ p_{\delta(V_H)})$ still improves $x$, as follows:
\begin{align*}
& \la \theta, p^y(x) \ra - \la \theta, x \ra \\
& = \sum_{e \in \delta(V_H) \cap E^+} \theta_e(1-x_e) + \sum_{e \in \delta(V_H) \cap E^-} \theta_e(1 - x_e) \nonumber \\
& \quad - \sum_{e \in E_H} \theta_e x_e \\
& \leq \sum_{e \in \delta(V_H) \cap E^+} \theta_e - \sum_{e \in \delta(V_H) \cap E^+} \theta_e x_e - \sum_{e \in E_H} \theta_e x_e \\
& \leq \sum_{e \in \delta(V_H) \cap E^+} \theta_e - \sum_{e \in \delta(V_H) \cap E^+} \theta_e \\
& = 0. \qedhere
\end{align*}
\end{proof}

The inequality \eqref{eq:subgraph-multicut-condition-with-boundary} is less restrictive than \eqref{eq:subgraph-multicut-condition} in Theorem~\ref{thm:multicut-subgraph-criterion}, because the left-hand side is potentially larger. Indeed, if two neighboring nodes $u, v \in V_H$ are connected by positive edges to some vertex $w \in \partial V_H$ in the boundary (i.e.\ they form a positive triangle), then the extension of any cut that separates $u$ from $v$ has to cut another edge of the triangle. Thus, the weight of this edge can be subtracted from the right-hand side of the inequality \eqref{eq:subgraph-multicut-condition} or, equivalently, added to the left-hand side, which is what \eqref{eq:subgraph-multicut-condition-with-boundary} achieves.

For the special case of a single edge subgraph $H = (\{u,v\},\{uv\})$ the refined condition is explicitly stated as
\begin{align}
\theta_{uv} & \geq \sum_{e \in \delta({uv}) \cap E^+} \theta_e - \hspace{-0.5em} \sum_{w \neq u,v \mid uw, vw \in E^+} \hspace{-0.5em} \min \{ \theta_{uw}, \theta_{vw} \}.
\end{align}